\documentclass[12pt]{amsart}
\usepackage[all]{xy}
\usepackage{amssymb}
\usepackage{epigraph}
\usepackage{amsthm}
\usepackage{hyperref}
\hypersetup{colorlinks=true,linkcolor=blue,citecolor=magenta}
\usepackage{amsmath}
\usepackage{amscd,enumitem}
\usepackage{verbatim}
\usepackage{eurosym}
\usepackage{float}
\usepackage{graphicx}
\usepackage[section]{placeins}
\usepackage{color}
\usepackage{dcolumn}
\usepackage[mathscr]{eucal}
\usepackage[all]{xy}
\usepackage{hyperref}
\usepackage{bbm}
\usepackage[textheight=8.5in, textwidth=6.7in]{geometry}
\usepackage{multirow}
\usepackage{caption}
\newtheorem*{thm*}{Theorem}
\newtheorem*{conj*}{Conjecture}

\newtheorem*{remark}{Remark}

\newtheorem{thm}{Theorem}[section]
\newtheorem{lem}{Lemma}[section]
\newtheorem{cor}[thm]{Corollary}

\newtheorem{prop}[thm]{Proposition}

\newtheorem{definition}[thm]{Definition}

\newcommand{\ord}{\mathrm{ord}}
\newcommand{\Z}{\mathbb{Z}}
\newcommand{\Q}{\mathbb{Q}}

\newcommand{\N}{\mathbb{N}}
\newcommand{\Mod}[1]{\ (\mathrm{mod}\ #1)}
\newcommand{\F}{\mathbb{F}}

\newcommand{\tor}{\mathrm{tor}}

\mathchardef\mathcomma=\mathcode`,
\newtheorem*{theorem*}{Theorem}

\makeatletter
\AtBeginDocument{%
\expandafter\renewcommand\expandafter\subsection\expandafter{%
\expandafter\@fb@secFB\subsection
}%
}
\makeatother

\numberwithin{equation}{section}

\makeatletter
\newcommand\footnoteref[1]{\protected@xdef\@thefnmark{\ref{#1}}\@footnotemark}
\makeatother

\begin{document}
\title{A Short Note on Inadmissible Coefficients of Weight $2$ and $2k+1$ Newforms}
\author{Malik Amir and Andreas Hatziiliou }
\address{Department of Mathematics, \'Ecole Polytechnique F\'ed\'erale de Lausanne}
\email{malik.amir.math@gmail.com}
\address{Department of Mathematics and Statistics, McGill University}
\email{andreas.hatziiliou@mail.mcgill.ca}
\keywords{Lucas Sequences, Lehmer's Conjecture, Modular Forms, $L$-Functions, Elliptic Curves}

\begin{abstract} 

Soit $f(z)=q+\sum_{n\geq 2}a(n)q^n$ une forme de Hecke normalisée de poids $k$ à coefficients entiers possédant une représentation galoisienne modulo 2 triviale. On généralise les résultats de Amir et Hong présentés dans \cite{AH} pour le poids $k=2$ en éliminant ou en localisant toutes les valeurs impaires $|\ell|<100$ des coéfficient de Fourier $a(n)$ pour $n$ respectant certaines congruences. On étudie aussi le cas des poids impairs $k\geq 1$ de formes de Hecke dont le caractère est donné par un caractère quadratique de Dirichlet.

Let $f(z)=q+\sum_{n\geq 2}a(n)q^n$ be a weight $k$ normalized newform with integer coefficients and trivial residual mod 2 Galois representation. We extend the results of Amir and Hong in \cite{AH} for $k=2$ by ruling out or locating all odd prime values $|\ell|<100$ of their Fourier coefficients $a(n)$ when $n$ satisfies some congruences. We also study the case of odd weights $k\geq 1$ newforms where the nebentypus is given by a quadratic Dirichlet character.
\end{abstract}
\epigraph{It’s not the violence of the few that scares me, it’s the silence of the many.}{\textit{Palestine will be free}}
\maketitle
\section{Introduction and Statement of the Results}
In an article entitled \textit{``On certain arithmetical functions''} \cite{Ram}, Ramanujan introduced the $\tau$-function in 1916, known as the Fourier coefficients of the weight $12$ modular form $$\Delta(z)=q\prod_{n=1}^\infty (1-q^n)^{24}:=\sum_{n=1}^\infty \tau(n)q^n=q-24q^2+252q^3-1472q^4+4830q^5-...$$ where throughout $q:=e^{2\pi i z}$. It was conjectured by Ramanujan that the $\tau$-function is multiplicative and this offered a glimpse into a much more general theory known today as the theory of Hecke operators. Despite its importance in the large web of mathematics and physics, basic properties of $\tau(n)$ are still unknown. The most famous example is Lehmer's conjecture about the nonvanishing of $\tau(n)$. Lehmer proved that if $\tau(n)=0$, then $n$ must be a prime \cite{Lehmer}.  One may be interested in studying odd values taken by $\tau(n)$ or, more generally, the coefficients of any newform. This is the question we consider as a variation of Lehmer's original speculation.

For an odd number $\alpha$, Murty, Murty and Shorey \cite{MMS} proved using linear forms in logarithms that $\tau(n)\neq\alpha$ for all $n$ sufficiently large. However, the bounds that they obtained are huge and computationally impractical. Recently, using the theory of Lucas sequences, Balakrishnan, Craig, Ono and Tsai proved in \cite{BCO} and \cite{BCOT}, together with work of Dembner and Jain in \cite{DJ}, that $\tau(n)=\ell$ has no solution for $|\ell|<100$ an odd prime. In addition, Hanada and Madhukara proved in \cite{HM} that $\tau(n)=\alpha$ has no solution for $|\alpha|<100$ an odd integer. Following these ideas, Amir and Hong investigated weight $2$ and $3$ newforms corresponding to modular elliptic curves and a special family of $K3$ surfaces in \cite{AH}. 

In this paper, we extend slightly the results of \cite{AH} on inadmissible coefficients for $L$-functions of modular elliptic curves and give a procedure to rule out odd prime values $\ell$, positive or negative, as coefficients of any normalized newform of odd weight $k\geq 1$ with integer coefficients having trivial residual mod $2$ Galois representation and a quadratic Dirichlet character. For the rest of this paper, whenever we say newform of weight $k$, we talk about a newform with the aforementioned properties. In the case of weight $2$ newforms, we have the following results. 

\begin{thm}\label{thmInadmissible}
Suppose $f(z)=q+\sum_{n\geq 2}a(n)q^n\in S_2^{new}(\Gamma_0(N))\cap \Z[[q]]$ has trivial residual mod 2 Galois representation, namely, $E/\Q$ is an elliptic curve of conductor $N$ with a rational $2$-torsion point. Then the following are true.  
\begin{enumerate}
\item[] If $E/\Q$ has a rational $3$-torsion point, then for $n>1$ and $\gcd(n,2\cdot 3\cdot N)=1$, we have 
\begin{itemize}
\item [1.] If $a(n)=7,13,19,31,37$, then $n=p^2$ and $p=2 \Mod{3}$.
\item[2.]If $a(n)=29$ then $n=p^{d-1}=13^4$ and $a(p)=\pm 2$.
\item[3.]If $a(n)=41$ then $n=p^{d-1}=43^4$ and $a(p)=\pm 4$.
\item[4.]If $a(n)=-19$ then $n=p^{d-1}=7^4$ and $a(p)=\pm 2$.
\item[5.] If $a(n)=-31$ then $n=p^{d-1}=7^4$ and $a(p)=\pm 4$.
\item[6.]If $a(n)=-79$ then $n=p^{d-1}=167^4$ and $a(p)=\pm 8$.
\end{itemize}
Furthermore, $$a(n)\not \in \{-1,1,5,-7,11,-13, 17,23,-37,-43,47,53,59,-61,-67,71,-73,83,89,-97\}.$$

\item[] If $E/\Q$ has a rational $5$-torsion point, then for $n>1$ and $\gcd(n,2\cdot5\cdot N)=1$, we have
\begin{itemize}
\item[1.] If $\ell\equiv 1\Mod{5}$ and $a(n)=\ell$, then $n=p^2$ and $p\equiv 4\Mod{5}$. 
\item[2.] If $\ell\equiv 2\Mod{5}$, $\ell\neq -3$ and $a(n)=\ell$, then $n=p^2$ and $p\equiv 2\Mod{5}$. 
\item[3.] If $\ell\equiv 3\Mod{5}$, $\ell\neq 3$ and $a(n)=\ell$, then $n=p^2$ and $p\equiv 1,3\Mod{5}$.
\end{itemize}
Furthermore, $$a(n)\not \in \{-1,1,-11,19,29,-31,-41,59,-61,-71,79,89,-691\}. $$
\end{enumerate}
\end{thm}
\begin{thm}
Let $E/\Q$ be an elliptic curve of conductor $N$ with a $2$ and $3$-torsion point. Let $n>1$ and $gcd(n,2\cdot 3\cdot N)=1$. If $\ell\equiv 2\Mod{3},\ell\neq 5$ and the odd prime divisors $d$ of $|\ell|(|\ell|-1)(|\ell|+1)$ are not congruent to $2\Mod{3}$, then $a(n)\neq \ell$. 
\end{thm}
\begin{thm}
Let $E/\Q$ be an elliptic curve of conductor $N$ with a $2$ and $5$-torsion point. Let $n>1$ and $\gcd(n,2\cdot 5\cdot N)=1$.
\begin{itemize}
\item[1.]  If $\ell\equiv 1\Mod{5}$ and the odd prime divisors $d$ of $|\ell|(|\ell|-1)(|\ell|+1)$ are not congruent to $1,3\Mod{5}$, then $a(n)\neq \ell$. 
\item[2.] If $\ell\equiv 2\Mod{5}, \ell\neq -3$ and the odd prime divisors $d$ of $|\ell|(|\ell|-1)(|\ell|+1)$ are not congruent to $2,3\Mod{5}$, then $a(n)\neq \ell$. 
\item[3.] If $\ell\equiv 3\Mod{5},\ell\neq 3$ and the odd prime divisors $d$ of $|\ell|(|\ell|-1)(|\ell|+1)$ are not congruent to $2,3\Mod{5}$, then $a(n)\neq \ell$. 
\item[4.] If $\ell\equiv 4\Mod{5}$ and the odd prime divisors $d$ of $|\ell|(|\ell|-1)(|\ell|+1)$ are not congruent to $2,4\Mod{5}$, then $a(n)\neq \ell$. 
\end{itemize}
\end{thm}
\begin{thm}{\label{thm3.8}} Let $E / \Q$ be an elliptic curve with conductor $N$ and $f$ the corresponding newform with Fourier coefficients $a(n)$. For $r = 3,5,$ suppose that $2\cdot r$ divides $|E_{\tor}(\mathbb{Q})|$. Then $a(p^{d-1})\neq r^v$ unless $d=r$ for some $v\in \N$.
\end{thm}

For odd weights $k\geq 3$ newforms, we have the following result. 
\begin{thm}{\label{not1weightodd}}
Let $\gcd(n,2\cdot N)=1$. Then $a(p^{d-1})\neq \pm 1$ and for $n>1$, we also have $a(n)\neq \pm 1$. Furthermore, if $a(n)=\pm \ell$ for some prime $\ell$, then $n=p^{d-1}$ where $d|\ell(\ell^2-1)$ is odd. If $\pm \ell$ is not defective, then $d$ is an odd prime.
\end{thm}
In Section \ref{Section3.2} and \ref{Section4}, we give results allowing us to state the above theorems independently of the level.

\section{Preliminaries}\label{S2}
\subsection{Lucas Sequences and their primitive prime divisors}
We recall the deep work of Bilu, Hanrot and Voutier \cite{BHV} on Lucas sequences which is central to this note.

A \textit{Lucas pair} $(\alpha,\beta)$ is a pair of algebraic integers, roots of a monic quadratic polynomial $F(x)=(x-\alpha)(x-\beta)\in \Z[x]$ where $\alpha+\beta$, $\alpha \beta$ are coprime non-zero integers and such that $\alpha/\beta$ is not a root of unity. To any Lucas pair $(\alpha,\beta)$ we can associate a sequence of integers $\{u_n(\alpha,\beta)\}=\{u_1=1, u_2=\alpha+\beta,\dots\}$ called \textit{Lucas numbers} defined by the following formula  
\begin{equation}
u_n(\alpha,\beta):=\frac{\alpha^n-\beta^n}{\alpha-\beta}.
\end{equation}
We call a prime  $\ell \mid u_{n}(\alpha,\beta)$ a {\it primitive prime divisor of $u_n(\alpha,\beta)$} if $\ell \nmid (\alpha-\beta)^2 u_1(\alpha,\beta)\cdots u_{n-1}(\alpha, \beta)$. We call a Lucas number $u_n(\alpha,\beta)$ with $n>2$ {\it defective}\footnote{We do not consider the absence of
a primitive prime divisor for $u_2(\alpha,\beta)=\alpha+\beta$ to be   a defect.}  if $u_{n}(\alpha,\beta)$ does 
not have a primitive prime divisor. Bilu, Hanrot, and Voutier \cite{BHV} proved the following theorem for all Lucas sequences.
\begin{thm}\label{Bilu}
Every Lucas number $u_n(\alpha,\beta)$, with $n>30,$
has a primitive prime divisor.
\end{thm}
This theorem is sharp in the sense that there are sequences for which $u_{30}(\alpha,\beta)$
does not have a primitive prime divisor. Their work, combined with a subsequent paper\footnote{This paper included a few cases which were omitted in \cite{BHV}.} 
of Abouzaid \cite{Abouzaid}, gives the {\it complete classification} of
defective Lucas numbers in two categories; a sporadic family of examples and a set of infinite parametrized families, as can be seen from Tables 1-4 in Section 1 of \cite{BHV} and Theorem 4.1 of \cite{Abouzaid}. The main arguments in our proofs will largely rely on relative divisibility properties of Lucas numbers. We now recall some of these facts\footnote{See Section 2 of \cite{BHV}.}.

\begin{prop}[Prop. 2.1 (ii) of \cite{BHV}]\label{PropA}  If $d\mid n$, then $u_d(\alpha, \beta) | u_n(\alpha,\beta).$
\end{prop}

In order to keep track of the first occurrence of a prime divisor, we define $m_{\ell}(\alpha,\beta)$ to be the smallest $n\geq 2$ for which $\ell \mid u_n(\alpha,\beta)$. We note that $m_{\ell}(\alpha,\beta)=2$ if and only if
$\alpha +\beta\equiv 0\pmod {\ell}.$
\begin{prop}[Cor. 2.2\footnote{This corollary is stated for Lehmer numbers. The conclusions hold for Lucas numbers because $\ell \nmid (\alpha+\beta)$.} of \cite{BHV}]\label{PropB} If $\ell\nmid \alpha \beta$ is an odd prime with
$m_{\ell}(\alpha,\beta)>2$, then the following are true.
\begin{enumerate}
\item[1.] If $\ell \mid (\alpha-\beta)^2$, then $m_{\ell}(\alpha,\beta)=\ell.$
\item[2.] If $\ell \nmid (\alpha-\beta)^2$, then $m_{\ell}(\alpha,\beta) \mid (\ell-1)$ or $m_{\ell}(\alpha,\beta)\mid (\ell+1).$
\end{enumerate}
\end{prop}
\begin{remark}
If $\ell|\alpha\beta$, then either $\ell|u_n(\alpha,\beta)$ for all $n$ or $\ell\nmid u_n(\alpha,\beta)$ for all $n$. 
\end{remark}

We now recall the following facts about newforms of weight $k\in \N$ and character $\chi$ that we will denote by $S_k^{new}(\Gamma_0(N),\chi)$. We suggest that the reader takes a look at the book of Cohen and Strömberg \cite{CoSt} for a thorough introduction to the theory of modular forms and to the book of Ono \cite{Ono} for a clear and concise exposition to more advanced topics. 
\begin{prop}\label{Newforms} Suppose that $f(z)=q+\sum_{n\geq2} a(n)q^n\in S_{k}(\Gamma_0(N),\chi)$ is a normalized newform with nebentypus $\chi$.
Then the following are true.
\begin{enumerate}
\item[1.] If $\gcd(n_1,n_2)=1,$ then $a(n_1 n_2)=a(n_1)a(n_2).$
\item[2.] If $p\nmid N$ is prime and $m\geq 2$, then $$a(p^m)=a(p)a(p^{m-1})-\chi(p)p^{k-1}a(p^{m-2}).
$$
\item[3.] If $p\nmid N$ is prime and $\alpha_p$ and $\beta_p$ are roots of $F_p(x):=x^2-a(p)x+\chi(p)p^{k-1},$ then
$$
a(p^m)=u_{m+1}(\alpha_p,\beta_p)=\frac{\alpha_p^{m+1}-\beta_p^{m+1}}{\alpha_p-\beta_p}.
$$   
Moreover, we have the Deligne's bound $|a(p)|\leq 2p^{\frac{k-1}{2}}$.
\end{enumerate}
\end{prop}
In this note, we consider Lucas sequences arising from the roots of the Frobenius polynomial 
\begin{equation}{\label{Frob}}
F_p(x):=x^2-Ax+B:=x^2-a(p)x+\chi(p)p^{k-1}=(x-\alpha_p)(x-\beta_p),
\end{equation} 
for a fixed prime $p\nmid N$ where $$u_{n}(\alpha_p,\beta_p):=a(p^{n-1})=\frac{\alpha_p^{n}-\beta_p^{n}}{\alpha_p-\beta_p},$$ 
and $|a(p)|\leq2p^{\frac{k-1}{2}}$. 

\subsection{Modular Forms and their Galois Representation}
\begin{definition}
We say that a newform $f\in S_k^{new}(\Gamma_0(N),\chi)$ has trivial residual mod 2 Galois representation if $a(p)$ is even for all $p\nmid 2\cdot N$.
\end{definition}
\begin{remark}
The condition $p\neq 2$ comes from the fact that the determinant of the representation of the Galois group evaluated at the Frobenius element needs to be nonzero in order to be invertible. Using Proposition $\ref{Newforms}$-\textit{(2)}, this implies that we can derive $a(p^d)$ to be odd if and only if $d$ is even. Similarly, we get that $a(p^d)$ is even if and only if $d$ is odd. It follows that $a(n)$ is odd if and only if $n$ is an odd square. Furthermore, requiring $f\in S_2^{new}(\Gamma_0(N))$ to have trivial residual mod $2$ Galois representation is equivalent to asking that the associated modular elliptic curve has a rational $2$-torsion point. 
\end{remark}
\section{Newforms of weight $k=2$}
We begin by studying newforms of weight $k=2$, level $N$ and trivial character $\chi$ with integer coefficients. These modular forms are interesting by themselves as they correspond to modular elliptic curves. For completeness, we recall some facts from the weight 2 case presented in \cite{AH}.

\begin{lem}[Lemma 2.1 \cite{AH}]\label{nondefect}
Assume $a(p)$ is even for primes $p\nmid 2\cdot N$. The only defective odd values $u_d(\alpha_p,\beta_p)$ are given in Table 1 by $$(d,A,B,k)\in\{(3,2,3,2),(5,2,11,2)\},$$ and in rows 1 and 2 of Table 2.
\end{lem}

\begin{remark}
For $u_3(\alpha_p,\beta_p)=\varepsilon 3^r$, we have $3\nmid a(p)$ and so $u_3(\alpha,\beta)$ is the first occurrence of $3$ in the sequence. If $3|a(p)$ however, then $\varepsilon 3^r$ is no longer a defective value. 
\end{remark}

We will make use of the two following fundamental results.
\begin{thm}[Modularity Theorem \cite{Modularity}]\label{Modularity}
Let $E / \Q$ be an elliptic curve with conductor $N$ and a rational $2$-torsion point. Denote the associated newform with trivial residual mod $2$ Galois representation by $f_E(z) = \sum_{ n\geq 1} a(n) q^n \in S_2^{New}(\Gamma_0(N))\cap \Z[[q]]$. Then for all primes $p\nmid 2\cdot N$ of good reduction, we have
$$ a(p) = p+1 - \#E(\F_p), $$
where $\#E(\F_p)$ denotes the number of $\F_p$-points of the elliptic curve reduced mod $p$.
\end{thm}

The following theorem of Mazur classifies the possible torsion groups of ellitpic curves of $E/\Q$.
\begin{thm}[Mazur's Theorem \cite{Maz}]{\label{Mazur}}
If $E/\Q$ is an elliptic curve, then $$E_{\tor}(\Q)\in \{\Z/N\Z\ |\ 1\le N\le 10\ \mathrm{or}\ N=12\}\cup\{\Z/2\Z \times \Z/N\Z \ |\ N=2,4,6,8\}.$$
\end{thm}
Furthermore, recall that if $E / \Q$ has good reduction at $p\nmid m$ for some $m\in \N$, then the reduction map 
\begin{equation}\label{injectivityoftorsion}
\pi \colon E(\Q) \to E(\F_p),
\end{equation}
is injective when restricted to $m$-torsion \cite{Silverman}. As a consequence, we get the following result.
\begin{lem}[Lemma 3.1 \cite{AH}]\label{lemma1}\label{Initial congruence}
Suppose that $E / \Q$ is an elliptic curve and that $r \mid\#E_{\tor}(\Q)$. Then for all primes $p\nmid 2\cdot r\cdot N$, we have
$$a(p^d) \equiv 1+p+p^2+\cdots +p^d \mod r. $$
\end{lem}



\begin{lem}[Lemma 3.2 \cite{AH}]{\label{not1}}
If $E/\Q$ has a rational $2$ and $r$-torsion point where $r=3,5$, then for all $\gcd(n,2\cdot r\cdot N)=1$, we have $|a(n)|\neq 1$.
\end{lem}

\subsection{Integer points on Thue equations}{\label{IPSC}} We discuss the general approach to solve the equation $a(n)=\ell$ for some prime $\ell$, positive or negative, where $a(n)$ is the Fourier coefficient of $f\in S_k^{new}(\Gamma_0(N),\chi)\cap \Z[[q]]$. Assume that $|a(n)|\neq 1$ for some given values of $n$. Using Proposition $\ref{Newforms}$\textit{ (1)-(2)}, we can see that studying $a(n)=\ell$ is equivalent to studying $a(p^{d-1})=\ell$ for $d\big||\ell|(|\ell|-1)(|\ell|+1)$. From the two-term recurrence relation satisfied by $a(p^{d-1})$, $a(p^{d-1})=\ell$ reduces to the search of integer points on special curves. We make this statement precise now. 

Let $D$ be a non-zero integer. A polynomial equation of the form $F(X,Y)=D$, where $F(X,Y)\in \Z[X,Y]$ is a homogeneous polynomial, is called a \textit{Thue} equation. We will consider those equations arising from the series expansion of 
\begin{equation}
\label{Thue}
\frac{1}{1-\sqrt{Y}T+XT^2}=\sum_{m=0}^\infty F_m(X,Y)\cdot T^m=1+\sqrt{Y}\cdot T+(Y-X)\cdot T^2+\dots
\end{equation}
\begin{lem}
If $a(n)$ satisfies Proposition $\ref{Newforms}$, and $p\nmid N$ is a prime, then
$$F_{2m}\left(\chi(p)p^{k-1},a(p)^2\right)=a(p^{2m}).$$
\end{lem}
Hence, solving $a(p^{2m})=q$ boils down to computing integer solutions to the equation $$F_{2m}(X,Y)=\ell.$$

Methods for solving Thue equations are implemented in Sage \cite{github}, Magma, and are best suited for $m\geq 3$. For $m=1,2$, these equations often have infinitely many solutions as they do not represent curves with positive genus when the weight is $k=1,2$, hence we require extra information to infer their finiteness. In the case of weight $2$ newforms, the idea is to introduce a $3$ and $5$-torsion point to get additional congruences to avoid having to deal with infinitely many solutions for the equations $a(p^2)=\ell,\,a(p^4)=\ell$. Indeed, note that $$d\big ||\ell|(|\ell|-1)(|\ell|+1)\equiv 0\Mod{3},$$ for all $\ell$ and hence $d=3$ will always have to be checked. 
\subsection{Some Congruences}{\label{Section3.2}}
We now list congruences obtained using Lemma $\ref{Initial congruence}$.
\begin{lem}{\label{congruences}}
Let $E/\Q$ be an elliptic curve with conductor $N$ having a rational $2$ and $3$-torsion point. Consider primes $p$ for which $\gcd(p,2\cdot 3\cdot N)=1$.
\begin{itemize}
\item[1.]If $a(p^{d-1})=\ell=\pm 3$, 
then $$(p,d)\in\{(1,0),(2,0),(2,2)\Mod{3}\}.$$
\item[2.] If $a(p^{d-1})=\ell\equiv 1\Mod{3}$, then $$(p,d)\in \{(1,1),(2, \text{odd})\Mod{3}\}.$$ 
\item[3.]If $a(p^{d-1})=\ell\equiv 2[3]$, then $$(p,d)=(1,2)\Mod{3}.$$
\end{itemize}
\begin{remark}
In point 2, the last pair is problematic as $d$ is always odd. Hence, it is not possible to provide a general result in this case. 
\end{remark}
Let $E/\Q$ be an elliptic curve with conductor $N$ having a rational $2$ and $5$-torsion point. Consider primes $p$ for which $\gcd(p,2\cdot 5\cdot N)=1$.
\begin{itemize}
\item[1.] If $a(p^{d-1})=\pm 5$, then $$(p,d)\in \{(1,0),(3,0),(3,4),(4,0),(4,4),(4,2)\Mod{5}\}.$$
\item[2.] If $a(p^{d-1})=\ell\equiv 1\Mod{5}$, then $$(p,d)\in\{(1,1),(2,1),(3,1),(4,1),(4,3)\Mod{5}\}.$$
\item[3.] If $a(p^{d-1})=\ell\equiv 2\Mod{5}$, then $$(p,d)\in \{(1,2),(2,3)\Mod{5}\}.$$
\item[4.] If $a(p^{d-1})=\ell\equiv 3\Mod{5}$, then $$(p,d)\in\{(1,3),(3,3),(2,2)\Mod{5}\}.$$ 
\item[5.] If $a(p^{d-1})=\ell\equiv 4\Mod{5}$, then $$(p,d)\in \{(1,4),(3,2)\Mod{5}\}.$$
\end{itemize}

\end{lem}
Using the above congruences extensively as well as the relative divisibility property of Lucas numbers, we get the following result.
\begin{thm}\label{thmInadmissible}
Suppose $f(z)=q+\sum_{n\geq 2}a(n)q^n\in S_2^{new}(\Gamma_0(N))\cap \Z[[q]]$ has trivial residual mod 2 Galois representation, namely, $E/\Q$ is an elliptic curve of conductor $N$ with a rational $2$-torsion point. Then the following are true.  
\begin{enumerate}
\item[] If $E/\Q$ has a rational $3$-torsion point, then for $n>1$ and $\gcd(n,2\cdot 3\cdot N)=1$, we have 
\begin{itemize}
\item [1.] If $a(n)=7,13,19,31,37$, then $n=p^2$ with $p=2\Mod{3}$.
\item[2.]If $a(n)=29$ then $n=p^{d-1}=13^4$ and $a(p)=\pm 2$.
\item[3.]If $a(n)=41$ then $n=p^{d-1}=43^4$ and $a(p)=\pm 4$.
\item[4.]If $a(n)=-19$ then $n=p^{d-1}=7^4$ and $a(p)=\pm 2$.
\item[5.] If $a(n)=-31$ then $n=p^{d-1}=7^4$ and $a(p)=\pm 4$.
\item[6.]If $a(n)=-79$ then $n=p^{d-1}=167^4$ and $a(p)=\pm 8$.
\end{itemize}
Furthermore, $$a(n)\not \in \{-1,1,5,-7,11,-13, 17,23,-37,-43,47,53,59,-61,-67,71,-73,83,89,-97\}.$$

\item[] If $E/\Q$ has a rational $5$-torsion point, then for $n>1$ and $\gcd(n,2\cdot5\cdot N)=1$, we have
\begin{itemize}
\item[1.] If $\ell\equiv 1\Mod{5}$ and $a(n)=\ell$, then $n=p^2$ and $p\equiv 4\Mod{5}$. 
\item[2.] If $\ell\equiv 2\Mod{5}$, $\ell\neq -3$ and $a(n)=\ell$, then $n=p^2$ and $p\equiv 2\Mod{5}$. 
\item[3.] If $\ell\equiv 3\Mod{5}$, $\ell\neq 3$ and $a(n)=\ell$, then $n=p^2$ and $p\equiv 1,3\Mod{5}$.
\end{itemize}
Furthermore, $$a(n)\not \in \{-1,1,-11,19,29,-31,-41,59,-61,-71,79,89,-691\}. $$
\end{enumerate}
\end{thm}
\begin{remark}
In the first part of the theorem, we have omitted the primes $43,61,67,73,79,97$, which are of the form $\ell\equiv 1\Mod{3}$, and the primes $-\ell\equiv 1\Mod{3}$ due to the large number of curves involved. However, following the methods outlined in this text, the interested reader will have no difficulty investigating these cases. In the second part, note that the primes of the form $\ell\equiv 4\Mod{5}$ are those of the outlined list. We've also ruled out $-691$ simply because it is a pretty number and a nice example of application of Theorem \ref{thm3.5}.  
\end{remark}
\begin{thm}
Let $E/\Q$ be an elliptic curve of conductor $N$ with a $2$ and $3$-torsion point. Let $n>1$ and $gcd(n,2\cdot 3\cdot N)=1$. If $\ell\equiv 2\Mod{3},\ell\neq 5$ and the odd prime divisors $d$ of $|\ell|(|\ell|-1)(|\ell|+1)$ are not congruent to $2\Mod{3}$, then $a(n)\neq \ell$. 
\end{thm}
\begin{thm}{\label{thm3.5}}
Let $E/\Q$ be an elliptic curve of conductor $N$ with a $2$ and $5$-torsion point. Let $n>1$ and $\gcd(n,2\cdot 5\cdot N)=1$.
\begin{itemize}
\item[1.]  If $\ell\equiv 1\Mod{5}$ and the odd prime divisors $d$ of $|\ell|(|\ell|-1)(|\ell|+1)$ are not congruent to $1,3\Mod{5}$, then $a(n)\neq \ell$. 
\item[2.] If $\ell\equiv 2\Mod{5},\ell\neq -3$ and the odd prime divisors $d$ of $|\ell|(|\ell|-1)(|\ell|+1)$ are not congruent to $2,3\Mod{5}$, then $a(n)\neq \ell$. 
\item[3.] If $\ell\equiv 3\Mod{5},\ell\neq 3$ and the odd prime divisors $d$ of $|\ell|(|\ell|-1)(|\ell|+1)$ are not congruent to $2,3\Mod{5}$, then $a(n)\neq \ell$. 
\item[4.] If $\ell\equiv 4\Mod{5}$ and the odd prime divisors $d$ of $|\ell|(|\ell|-1)(|\ell|+1)$ are not congruent to $2,4\Mod{5}$, then $a(n)\neq \ell$. 
\end{itemize}
\end{thm}
The above theorems can be made independent of the level using the following lemma.
\begin{lem}{\label{NIndep}}
Let $p|N$ be a prime and $N$ the level of the newform $f(z)$. Then $$
a_f(p^m)=a(p)a(p^{m-1})=\begin{cases} (\pm 1)^m \ \ \ \ \  &{\text {\rm if}}\ \ord_p(N)=1,\\
0 \ \ \ \ \ &{\text {\rm if}}\ \ord_p(N)\geq 2.
\end{cases}
$$
\end{lem}

\begin{thm}{\label{thm3.8}} Let $E / \Q$ be an elliptic curve with conductor $N$ and $f(z)$ the corresponding newform with Fourier coefficients $a(n)$. For $r = 3,5,$ suppose that $2\cdot r$ divides $|E_{\tor}(\mathbb{Q})|$. Then $a(p^{d-1})\neq r^v$ unless $d=r$ for some $v\in \N$.
\end{thm}
\begin{proof}
Let $r=3$. Then by Lemma \ref{congruences}, $3\big| |a(p^{d-1})|$ if and only if $3|d$. Indeed, if $p\equiv 0, 2\Mod{3}$, then $a(p^{d-1})\equiv 1 \Mod{3}$ and so $p\equiv 1 \Mod{3}$ implies that $3|d$. Suppose that $a(p^{d-1})=3^v$ and $d>3$, then $a(p^{d-4})$ is also a multiple of $3$, which contradicts that $3^v$ is not defective. Thus $d=3$ is the only solution.

Let $r=5$. Then by Lemma \ref{congruences}, $5\big| |a(p^{d-1})|$ if and only if $p\equiv 1[5]$ and $5|d$. For $d>5$, we have $a(p^{d-6})$ is also a multiple of $5$, violating that $a(p^{d-1})$ is not defective. Hence $d=5$.
\end{proof}

\section{Newforms of odd weight $k\geq 3$}{\label{Section4}}
In this section, we explain the basic framework to rule out odd prime values $\pm \ell$ as Fourier coefficients of odd weight $k\geq 3$ normalized newforms with integer coefficients of level $N$ and nebentypus $\chi$ given by a quadratic Dirichlet character and trivial residual mod $2$ Galois representation.

\begin{thm}{\label{not1weightodd}}
Let $\gcd(n,2\cdot N)=1$. Then $a(p^{d-1})\neq \pm 1$ and for $n>1$ we have $a(n)\neq \pm 1$. It follows that if $a(n)=\pm \ell$ for some prime $\ell$, then $n=p^{d-1}$ where $d|\ell(\ell^2-1)$ is odd. If $\pm \ell$ is not defective, then $d$ is an odd prime.
\end{thm}
\begin{proof}
5Note that $\pm 1$ is a defective value which must be located in the first 2 rows of Table 4. In fact, the table of sporadic values does not have to be considered. In row 1 of Table 4, we have that if $a(p^{d-1})=-1$, then it must be at $u_3=a(p^2)=-1$, where the constraints imply that we must satisfy $\chi(p)p^{k-1}=a(p)^2+1$. Since $k-1$ is even, let's write $k-1=2m$. Then we're left with the equation $\chi(p)x^{2m}=y^2+1$ and clearly there are no solutions for $\chi(p)=0,-1$. If $\chi(p)=1$ then we obtain $(x^m)^2-y^2=1$ and this gives us the integer solutions $x=\pm 1,y=0$ which aren't allowed. Hence $a(p^{d-1})\neq -1$ for $\gcd(p,2\cdot N)=1$. Now for row 2, if $a(p^{d-1})=1$ then it must happen for $u_3=a(p^2)=1$ with constraints given by $\chi(p)p^{k-1}=a(p)^2-1$,$a(p)>1$. There are no solutions to this equation for $\chi(p)=0,-1,1$. Hence $a(p^{d-1})\neq \pm 1$. 
\end{proof}
\begin{lem}
The curve $a(p^2)=\pm \ell$ has no solutions if $\chi(p)=0$. If $\chi(p)=1$ then the curve has the form $(y-x^m)(y+x^m)=\pm \ell$ which has a unique solution depending on $\ell$ only. For $\chi(p)=-1$ the curve has the form $y^2+(x^m)^2=\pm \ell$ which has no solutions for $-\ell$ and has finitely many solutions for $+\ell$ as it is a sum of squares. 
\end{lem}
Hence, to rule out or locate any odd prime value $\ell$ as a Fourier coefficient of $f(z)$, it suffices to follow the following steps. Let $\gcd(n,2\cdot N)=1$ and $\ell$ be an odd prime.
\begin{center}

\begin{itemize}
\item [1.] By multiplicativity of the Fourier coefficients, we have that $$a(n)=\pm \ell \text{ if and only if } \prod_i a(p_i^{d_i-1})=\pm \ell.$$ 
\item[2.] Use Theorem \ref{not1weightodd}.
\item[3.] Use the Sage Thue solver \cite{github} to solve $a(p^{d-1})=\pm \ell$ and analyze the solutions. 
\end{itemize}
\end{center}

\begin{lem}[Proposition 13.3.14 \cite{CoSt}]
Let $p|N$ and assume that $\chi$ can be defined modulo $N/p$ and let $f=\sum a(n)q^n\in S_k^{new}(\Gamma_0(N),\chi)$ be a normalized eigenform. Then 
\begin{itemize}
\item [1.] If $p^2|N$, then $a(p)=0$.
\item [2.] If $p^2\nmid N$, then $a(p)^2=\chi_1(p)p^{k-2}$ where $\chi_1$ is the character modulo $N/p$ equivalent to $\chi$. 
\end{itemize}

\end{lem}
Using this lemma, we easily obtain the following. 
\begin{cor}
Suppose that $\chi$ can be reduced modulo $N/p$ and call it $\chi_1$. Then $a(p^m)=0$ for all $m\geq 1$ and $a(n)=0$ for all $(n,N)\geq p$. 
\end{cor}
\begin{proof}
Recall that $a(p^m)=a(p)a(p^{m-1})$ and if $p^2|N$ the result follows immediately. For $p^2\nmid N$ we have $a(p)^2=\chi_1(p)p^{k-2}$. Since $k$ is odd, $k-2$ is odd and if $\chi_1(p)\neq 0$ then $a(p)$ cannot be an integer which is what we require for our newforms. Hence $a(p)=0$ is the only possibility. 
\end{proof}
\begin{lem}[Corollary 13.3.17 \cite{CoSt}] Let $p|N$ and assume that $\chi$ cannot be defined modulo $N/p$. If $f\in S_k^{new}(\Gamma_0(N),\chi)$ is a normalized eigenform, then $|a(p)|=p^\frac{k-1}{2}$.
\end{lem}
\begin{cor}
Let $p|N$ and assume that $\chi$ cannot be defined modulo $N/p$. Then $a(p^m)=(\pm 1)^m(p^\frac{k-1}{2})^m$.
\end{cor}



\section{Newforms of weight $k=1$}
In this section, we discuss briefly the problems arising in the case of weight $k=1$ newforms. Recall that Proposition \ref{PropA} is the main tool to determine if $d$ is prime or not.
However, we may have that $n|d$ and $u_n=-1$. From the first row of the defective table, we know that this may happen only if $n=3$, namely if $u_3=a(p^2)=-1$.  Thus, we would like to avoid $3|d$ but that is not possible as $\ell(\ell-1)(\ell+1)$ is always divisible by 3. This implies that we only have the information that $d|\ell(\ell-1)(\ell+1)$ is odd and no longer an odd prime. First, this leads us to a large amount of equations to verify in order to rule out an odd prime as a Fourier coefficient. Secondly, when $d=3$, the equation that we obtain is $$y^2=\pm\ell+\chi(p)$$ and is incredibly difficult to solve as it is still an open problem to determine if there are infinitely many primes of the form $\ell=y^2+1$.

\newpage

\section{Appendix}

\begingroup
\setlength{\tabcolsep}{3pt} 
\renewcommand{\arraystretch}{1.5}
\begin{center} 
\begin{table}[!ht]
\begin{tabular}{|c|c|}
\multicolumn{2}{c}{} \\ \hline
$(A,B)$ & Defective $u_n(\alpha, \beta)$ \\ \hline \hline
\multirow{2}{3.5em}{$(\pm 1,2^1)$} & $u_5 = -1$, $u_7 = 7$, $u_8 = \mp 3$, $u_{12} = \pm 45$, \\ & $u_{13} = -1$, $u_{18} = \pm 85$, $u_{30} = \mp 24475$ \\ \hline
$(\pm 1,3^1)$ & $u_5 = 1$, $u_{12} = \pm 160$ \\ \hline
$(\pm 1, 5^1)$ & $u_7 = 1$, $u_{12} = \mp 3024$ \\ \hline
$(\pm 2, 3^1)$ & $u_3 = 1$, $u_{10} = \mp 22$ \\ \hline
$(\pm 2,7^1)$ & $u_8 = \mp 40$ \\ \hline
$(\pm 2, 11^1)$ & $u_5 = 5$ \\ \hline
$(\pm 4, 5^1)$ & $u_6=\pm 44$\\ \hline
$(\pm 5, 7^1)$ & $u_{10} = \mp 3725$ \\ \hline
$(\pm 3, 2^3)$ & $u_3 = 1$ \\ \hline
$(\pm 5, 2^3)$ & $u_6 = \pm 85$ \\ \hline
\end{tabular}
\medskip
\captionof{table}{\textit{Sporadic family of defective $u_n(\alpha, \beta)$ satisfying equation \ref{Frob} in even weight $2k$ including $2k=2$ \cite{BCOT}.}} 
\label{table1}
\end{table}
\end{center}
\endgroup

\begingroup
\setlength{\tabcolsep}{3pt} 
\renewcommand{\arraystretch}{1.5}
\begin{center} 
\begin{table}[!ht]
\begin{tabular}{|c|c|}
\multicolumn{2}{c}{} \\ \hline
$(A,B)$ & Defective $u_n(\alpha, \beta)$ \\ \hline \hline
$(\pm 1,-1)$ & $u_5=5,u_{12}=\pm 144$\\ \hline \multirow{2}{3.5em}{$(\pm 1,2^1)$} & $u_5 = -1$, $u_7 = 7$, $u_8 = \mp 3$, $u_{12} = \pm 45$, \\ & $u_{13} = -1$, $u_{18} = \pm 85$, $u_{30} = \mp 24475$ \\ \hline
$(\pm 1, 2^2)$ & $u_{5}=5$,$u_{12}=\mp 231$ \\ \hline 
$(\pm 1,3^1)$ & $u_5 = 1$, $u_{12} = \pm 160$ \\ \hline
$(\pm 1, 5^1)$ & $u_7 = 1$, $u_{12} = \mp 3024$ \\ \hline
$(\pm 2, 3^1)$ & $u_{10} = \mp 22$ \\ \hline
$(\pm 2,7^1)$ & $u_8 = \mp 40$ \\ \hline
$(\pm 2, 11^1)$ & $u_5 = 5$ \\ \hline
$(\pm 5, 7^1)$ & $u_{10} = \mp 3725$ \\ \hline

\end{tabular}
\medskip
\captionof{table}{\textit{Sporadic family of defective $u_n(\alpha, \beta)$ satisfying equation \ref{Frob} in odd weight $k\geq 1$.}} 
\label{table1}
\end{table}
\end{center}
\endgroup

\noindent

\begingroup
\setlength{\tabcolsep}{3pt} 
\renewcommand{\arraystretch}{2.5}
\begin{center} 
\begin{table}[!ht]
\begin{small}
\begin{tabular}{|c|c|c|}
\hline
$(A,B)$ & Defective $u_n(\alpha, \beta)$ & Constraints on parameters \\ \hline \hline
$(\pm m, p)$ & $u_3 = -1$ & $m>1$ and $p = m^2+1$ \\ \hline
\multirow{2}{*}{}
$(\pm m, p^{2k-1})$ &
$u_3 = \varepsilon 3^r$ &
$\begin{aligned} &\ \ \ \ \ \ (p, \pm m)\in B_{1, k}^{r, \varepsilon} \text{ with } 3\nmid m,\\ 
&(\varepsilon,r,m)\neq (1,1,2),
\text{ and } m^2 \geq 4\varepsilon 3^{r-1} \end{aligned}$\\ \hline
{\color{black}$(\pm m, p^{2k-1})$} & {\color{black}$u_4 = \mp m$} & {\color{black}$(p,\pm m) \in B_{2,k}$ with $m > 1$ odd} \\ \hline
$(\pm m, p^{2k-1})$ & $u_4 = \pm 2{\color{black}\varepsilon}m$ & $\begin{aligned}(p,\pm m)\in B_{3,k}^\varepsilon &\text{ with } {\color{black}(\varepsilon, m) \not = (1,2)}\\ &\text{  and } m > 2
\text{  even}
\end{aligned}$ \\ \hline
$(\pm m, p^{2k-1})$ & 
$u_6 = {\color{black}\pm (-2)^rm(2m^2+(-2)^r)/3}$ &
$\begin{aligned}(p, &\pm m)\in B_{4,k}^r \text{ with } \gcd(m,6) = 1, \\
&{\color{black}(r,m) \not = (1,1)}, \text{ and } {\color{black}m^2 \geq (-2)^{r+2}}\end{aligned}$ \\ \hline
$(\pm m, p^{2k-1})$ & ${\color{black}u_6=\pm \varepsilon m(2m^2+3\varepsilon)}$ & $(p,\pm m)\in B_{5,k}^\varepsilon$ with $3\mid m$ and $m>3$\\ \hline
$(\pm m, p^{2k-1})$ & $u_6 = \pm 2^{r+1}{\color{black}\varepsilon} m(m^2 + 3 {\color{black}\varepsilon} \cdot 2^{r-1}) $ &$\begin{aligned} (p, \pm m)\in B_{6,k}^{r,\varepsilon}
\text{ with } m \equiv 3 \bmod{6} \\  \text{and } m^2 \geq 3 {\color{black}\varepsilon} \cdot 2^{r+2}\end{aligned}$ \\ \hline
\end{tabular}
\end{small}
\medskip
\captionof{table}{Parameterized family of defective $u_n(\alpha, \beta)$ satisfying equation \ref{Frob} in even weight $2k\geq 2$ \cite{BCOT}.
\label{table2}
\textit{\newline Notation: $m, k, r\in \Z^{+}$, $\varepsilon = \pm 1$, $p$ is a prime number.}}
\end{table}
\end{center}
\endgroup

\begingroup
\setlength{\tabcolsep}{3pt} 
\renewcommand{\arraystretch}{2.5}
\begin{center} 
\begin{table}[!ht]
\begin{small}
\begin{tabular}{|c|c|c|}
\hline
$(A,B)$ & Defective $u_n(\alpha, \beta)$ & Constraints on parameters \\ \hline \hline
$(\pm m, \chi(p)p^{k-1})$ & $u_3 = -1$ & $\chi(p)p^{k-1} = m^2+1$ \\ \hline
$(\pm m, \chi(p)p^{k-1})$ & $u_3 = 1$ & $\chi(p)p^{k-1} = m^2-1\text{ with }m>1$ \\ \hline
\multirow{2}{*}{}
$(\pm m, \chi(p)p^{k-1})$ &
$u_3 = \varepsilon 3^r$ &
$\begin{aligned} &\ \ \ \ \ \ \chi(p)p^{k-1}=m^2-\varepsilon 3^r\text{ with } 3\nmid m,\\ 
&(\varepsilon,r,m)\neq (1,1,2),
\text{ and } r>0 \end{aligned}$\\ \hline
$(\pm m,\chi(p)p^{k-1})$ & $u_4=\pm \varepsilon m$ & $2\chi(p)p^{k-1}=m^2-\varepsilon \text{ with }2\nmid m,m\neq 1$ \\ \hline
$(\pm m,\chi(p)p^{k-1})$ & $u_4=\pm 2\varepsilon m$ & $2\chi(p)p^{k-1}=m^2-2\varepsilon\text{ with }2|m, (\varepsilon,m)\neq(1,2)$ \\ \hline
$(\pm m, \chi(p)p^{k-1})$ &
$u_6 =(\pm 2m^3\pm m)/3$ &
$\begin{aligned} &\ \ \ \ \ \ 3\chi(p)p^{k-1}=m^2-1\text{ with } 3\nmid m>3 \end{aligned}$\\ \hline
$(\pm m,\chi(p)p^{k-1})$ & $u_6=\pm 2\varepsilon m^3\pm 3m$ & $3\chi(p)p^{k-1}=m^2-3\varepsilon \text{ with }3\mid m$ \\ \hline
$(\pm m, \chi(p)p^{k-1})$ &
$u_6 =(\pm 2m^3(-2)^r\pm m((-2)^r)^2)/3$ &
$\begin{aligned} &\ \ \ \ \ \ 12\chi(p)p^{k-1}=4m^2-(-2)^{r+2}\\ & r>0,m\equiv \pm 1[6], (r,m)\neq(1,1) \end{aligned}$\\ \hline
$(\pm m,\chi(p)p^{k-1})$ & $u_6=\pm2 m^3\varepsilon\cdot 2^r\pm 3m(2^r)^2$ & $12\chi(p)p^{k-1}=4m^2-3\cdot 2^{r+2}\varepsilon, r>0,m\equiv 3[6]$ \\ \hline
\end{tabular}
\end{small}
\medskip
\captionof{table}{Parameterized family of defective $u_n(\alpha, \beta)$ satisfying equation \ref{Frob} in odd weight.
\label{table2}
\textit{\newline Notation: $m, k, r\in \Z^{+}$, $\varepsilon = \pm 1$, $p$ is a prime number.}}
\end{table}
\end{center}
\endgroup


\begin{thebibliography}{99}







\bibitem{Abouzaid} M. Abouzaid, \emph{Les nombres de Lucas et Lehmer sans diviseur primitif},
J. Th\'eor. Nombres Bordeaux \textbf{18}, 299-313 (2006).


\bibitem{AH} Malik Amir and Letong Hong, \emph{On L-functions of modular elliptic curves and certain K3 surfaces}, https://doi.org/10.1007/s11139-021-00388-w, Ramanujan Journal (2021).


\bibitem{BCO} J. S. Balakrishnan, W.  Craig, and K. Ono, \emph{Variations of Lehmer's Conjecture for Ramanujan's tau-function}, ISSN 0022-314X, Journal of Number Theory (2020).

\bibitem{BCOT} J. S. Balakrishnan, W.  Craig, K. Ono, and W.-L. Tsai, \emph{Variants of Lehmer's speculation for newforms}, 	arXiv:2005.10354 [math.NT] (2020).

\bibitem{github} J. S. Balakrishnan, W. Craig, and K. Ono, \emph{Sage code},  \url{https://github.com/jbalakrishnan/Lehmer}.


\bibitem{BH96} Y. Bilu and G. Hanrot, \emph{Solving the Thue equations of high degree}, Journal of Number Theory, \textbf{60}, 373-392 (1996).

\bibitem{BHV} Y. Bilu, G. Hanrot, and P. M. Voutier, \emph{Existence of primitive divisors of Lucas and Lehmer numbers},
J. Reine Angew. Math.  \textbf{539}, 75-122 (2001).

\bibitem{Modularity} C. Breuil, B. Conrad, F. Diamond, R. Taylor, \emph{On the modularity of elliptic curves over Q: wild 3-adic exercises}, Journal of the American Mathematical Society \textbf{14} (4): 843–939 (2001).


\bibitem{CoSt} H. Cohen and F. Strömberg, \emph{Modular Forms, A Classical Approach}, American Mathematical Society, Graduate Studies in Mathematics 179 (2017).
\bibitem{DJ} S. Dembner and V. Jain, \emph{Hyperelliptic curves and newform coefficients}, arXiv:2007.08358 [math.NT] (2020), submitted for publication.
\bibitem{HM} 
M. Hanada and R. Madhukara, \emph{Fourier coefficients of level 1 Hecke eigenforms}, arXiv:2007.08683 [math.NT] (2020), submitted for publication.

\bibitem{Lehmer} D. H. Lehmer, \emph{The vanishing of Ramanujan's $\tau(n)$}, Duke Math. J.
\textbf{14}, 429-433 (1947).

\bibitem{LMFDB} The LMFDB Collaboration, The $L$-functions and Modular Forms Database, \emph{http://lmfdb.org/}, 2020,[Online, accessed July 2020].




\bibitem{MMS} V. K. Murty, R. Murty, and N. Shorey, \emph{Odd values of the Ramanujan tau function},
Bull. Soc. Math. France \textbf{115} (1987), 391-395.

\bibitem{Maz} B. Mazur \emph{Rational isogenies of prime degree (with an appendix by D. Goldfeld)}, Invent. Math., 44(2) : 129-162 (1978).

\bibitem{Ono} K. Ono, \emph{The Web of Modularity : Arithmetic of the Coefficients of Modular Forms and $q$-series}, American Mathematical Society, CBMS regional conference series in mathematics, ISSN 0160-7642; no. 102 (2003).


\bibitem{Ram} S. Ramanujan, \emph{On certain arithmetical functions},
Trans. Camb. Phil. Soc., 22, 159–184 (1916).


\bibitem{Ribet} K. Ribet, \emph{Galois representations attached to eigenforms with nebentypus}, Lecture Notes in Math., vol. 601, Springer-Verlag, New York, pp.17-51 (1976).




\bibitem{Silverman} J. H. Silverman, \emph{The Arithmetic of Elliptic Curves}, Second Edition, New York: SpringerVerlag (2009). 

\end{thebibliography}
\end{document}